\documentclass[12pt,reqno]{amsart}
\usepackage[utf8]{inputenc}
\usepackage[english]{babel}
\usepackage{comment}
\usepackage{amsmath,amssymb,amsfonts,amsbsy,amsthm,amscd,latexsym,graphicx}
\usepackage{empheq,textcomp}
\usepackage{xcolor}
\usepackage{hyperref}

\theoremstyle{definition}
\newtheorem{theorem}{Theorem} [section]
\newtheorem{corollary}[theorem]{Corollary}
\newtheorem{lemma}[theorem]{Lemma}
\newtheorem{proposition}[theorem]{Proposition}
\newtheorem{definition}[theorem]{Definition}

\newtheorem{remark}[theorem]{Remark}
\newtheorem{example}[theorem]{Example}

\numberwithin{equation}{section}

\AtEndDocument{
  \par
  \bigskip
  \begin{tabular}{@{}l@{}}%
  (Rohit Pai and Ivan Rocha) School of Mathematics, Georgia Institute of Technology,\\ Atlanta, GA, 30318, USA.\\ 
   
    \textit{E-mail address}: \texttt{rpai32@gatech.edu} and \texttt{irocha7@gatech.edu}\\
$~$\\

   (Pu-Ting Yu) Department of Mathematics, University of Oregon, Eugene, OR, 97403,\\ USA.\\ 
   
    \textit{E-mail address}: \texttt{putingyu@uoregon.edu}\\
  \end{tabular}}

\newcommand{\Ec}{{\mathcal{E}}}
\newcommand{\Fc}{{\mathcal{F}}}
\newcommand{\Gc}{{\mathcal{G}}}
\newcommand{\Wc}{{\mathcal{W}}}
\newcommand{\N}{\mathbb{N}}
\newcommand{\R}{\mathbb{R}}

\newcommand{\Z}{\mathbb{Z}}

\newcommand{\Oc}{{\mathcal{O}}}

\newcommand{\Eq}{\, = \,}
\newcommand{\EQ}{\, = \,}

\newcommand{\Le}{\, \le \,}

\newcommand{\Ge}{\, \geq \,}

\newcommand{\plus}{\, + \,}
\newcommand{\qeddef}{{\quad $\diamondsuit$}}

\newcommand{\bigabs}[1]{\bigl|\,#1\,\bigr|}

\newcommand{\ip}[2]{\langle\,#1,#2\,\rangle}

\newcommand{\Bigip}[2]{\Bigl\langle \,#1, \, #2 \,\Bigr\rangle}

\newcommand{\norm}[1]{\|\,#1\,\|}

\newcommand{\Bignorm}[1]{\Bigl\|\,#1\,\Bigr\|}

\newcommand{\bigparen}[1]{\bigl(\,#1\,\bigr)}
\newcommand{\Bigparen}[1]{\Bigl(\,#1\,\Bigr)}

\newcommand{\set}[1]{\{#1\}}
\newcommand{\bigset}[1]{\bigl\{#1\bigr\}}

\newcommand{\clspan}{{\overline{\text{span}}}}

\newcommand{\inN}{_{n\in\N}}

\begin{document}
\title{Woven Weighted Exponentials}
\author{Rohit Pai, Ivan Rocha and Pu-Ting Yu}

\subjclass[2020]{42C15}

\keywords{Weighted exponentials, Gabor system, Woven frames}

\date{\today}

\pagestyle{plain}
\maketitle
\begin{abstract}
  Let $f$ and $g$ be nonzero functions in $L^2([0,1])$. The \emph{woven weighted exponential system} (associated with $f$ and $g$) is defined by $$\Wc(f,g)=\bigset{\set{fe^{2\pi i nt}}_{n\in J} \cup \set{ge^{2\pi i nt}}_{n\in J^c}\,|\,J\subset\Z}.$$
We say that $\Wc(f,g)$ is \emph{wovenly complete}, (resp. \emph{wovenly minimal}, a \emph{woven frame}) if the weaving $\set{fe^{2\pi i nt}}_{n\in J} \cup \set{ge^{2\pi i nt}}_{n\in J^c}$ is complete, (resp. minimal, a frame) for all $J\subseteq \Z.$ 
In this paper, we study conditions that imply certain approximation properties of $\Wc(f,g)$, such as completeness, minimality and 
the frame property. We first provide a complete characterization of the woven weighted exponential systems that are wovenly complete.
We also show that $\Wc(f,g)$ is a woven frame if $f/g$ is strictly positive or strictly negative over $[0,1].$
Additionally, several counterexamples are provided to show that certain seemingly correct conditions do not imply the desired approximation properties of $\Wc(f,g).$ All results presented in this paper apply equivalently to systems of regular translates and Gabor systems 
at critical density in $L^2(\R)$.\\
\end{abstract}

\section{Introduction}
\label{introduction}
\emph{Frames} in separable Hilbert spaces provide a flexible generalization of orthonormal bases, allowing redundant but more stable expansions of vectors. From a theoretical perspective, frames make it possible to construct systems with specific structure (such as Gabor systems, wavelet systems) that might fail to form an orthonormal basis but still provide stable expansions for reconstructing every vector in the space. On the other hand, using frames as the foundation to reconstruct vectors have proved to be useful in numerous applications such as signal processing (\cite{FG98}, \cite{BHNS17}), denoising and inpainting (\cite{CCS08}, \cite{ZYZF16}), and sparse representation \cite{HHKHJ19}. Here we say that a sequence $\set{x_n}_{n\in I}$ in a separable Hilbert space is a \emph{frame} for $H$ if there exist some positive constants $A$ and $B$, called \emph{frame bounds}, such that 
\begin{equation}
\label{frame_ineq}
A\,\norm{x}^2_H \Le \sum_{n\in I}|\ip{x}{x_n}|^2\Le B\,\norm{x}^2_H\quad\text{for all }x\in H.
\end{equation}
We say that $\set{x_n}_{n\in \N}$ is a \emph{Bessel sequence} if at least the upper inequality of Equation (\ref{frame_ineq}) is satisfied.
For more history and theoretical aspect of frames, we refer to \cite{You01}, \cite{Heil11} and \cite{Chr18} for relatively recent textbook treatments.

The motivation behind the work presented in this paper arises from a practical technical challenges in distributed sensing, where one aims at recovering signal using multiple sampling devices. Similar challenges occur in other contexts in which measurements are obtained from distinct sources. As a result, the notion of \textit{woven frames} has emerged in recent years as a way of studying how two or more frames interact.  Let $\set{f_n}_{n \in I}$ and $\set{g_n}_{n \in I}$ be two sequences in $H$, indexed by the same set $I$. For a given subset $J \subseteq I$, the \textit{weaving} associated with $\set{f_n}_{n \in I}$, $\set{g_n}_{n \in I}$ and $J\subseteq I$ is the set defined by 
\begin{equation*}
    \Wc(J) \Eq \set{f_n}_{n \in J} \, \cup \, \set{g_n}_{n \in J^c}.
\end{equation*}
The \emph{woven system} associated with $\set{f_n}_{n \in I}$ and $\set{g_n}_{n \in I}$ is the collection of all weavings $$\Wc\bigparen{\set{f_n}_{n\in I},\set{g_n}_{n\in I}}\Eq\set{\Wc(J)\,|\,J\text{ is any subset of } \Z}.$$
Assume that both $\set{f_n}_{n \in I}$ and $\set{g_n}_{n \in I}$ are frames for $H$. Then the pair of frames $\set{f_n}_{n \in I}$ and $\set{g_n}_{n \in I}$ are called woven frames if every weaving in  $\Wc\bigparen{\set{f_n}_{n\in I},\set{g_n}_{n\in I}}$ is a frame for $H$. 
The notion of woven frames was first introduced in \cite{GrocCasa15} by Bemrose et al.\ and was further investigated in \cite{CFL16}, \cite{HD19}, \cite{VGDD18} and \cite{CabrelliMoulter24} (see also references therein). However, many questions along this line of research remain open. For example, only a small collection of frames is currently known to form woven frames. Beyond the frame property itself, we are also interested in determining under what conditions every weaving of two systems with certain approximation properties (see Section \ref{preliminaries} for details of these specific properties) still retains the same property. We investigate this question in this paper. 


We will particularly study our main question in the setting of \emph{systems of weighted exponentials} due to their connections to \emph{systems of regular translates} and \emph{regular Gabor system at critical density}. Let $f\in L^2[0,1]$. The associated system of weighted exponentials with $f$, denoted by $\Ec(f)$, is the sequence $\set{fe^{2\pi in\cdot x}}_{n\in\Z}$. Our main question is now formulated as follows: ``Assume that two systems of weighted exponentials $\Ec(f)$ and $\Ec(g)$ possess certain approximation properties. Under what conditions do we have that $\set{fe^{2 \pi i n\cdot x}}_{n \in J}\cup \set{ge^{2 \pi i n \cdot x}}_{n \in J^c}$ retains the same properties for every subset $J\subseteq\Z$?" It is known that regular Gabor systems in $L^2(\R)$ with critical density are unitarily equivalent to weighted exponential systems in $L^2[0,1]^{2}$ via \emph{Zak transform}. Moreover, any system of regular translates in $L^2(\R)$ can be unitarily mapped to a weight exponential system through the \emph{fiberization map} (see Section \ref{preliminaries} for the relevant terminology). As a result, all results proved for systems of weighted exponentials can be applied equivalently to regular Gabor systems with critical density and system of regular translates.

The major contributions of this paper are threefold. At first, we prove in Theorem \ref{woven_Riesz_suff_condi} that two frames of weighted exponentials, $\Ec(f)$ and $\Ec(g)$, form a woven frame for $L^2[0,1]$ if the sign of $f(t)$ and $g(t)$ are consistent for almost every $t\in[0,1].$ As an immediate consequence of this result, we obtain the following theorem.
\begin{theorem}
\label{sufficient_woven_frames_special_case}
    Let $f,g\in L^2[0,1]$ be real-valued continuous functions such that $\Ec(f)$ and $\Ec(g)$ are frames for $L^2[0,1].$ Then $\Ec(f)$ and $\Ec(g)$ form a woven frame for $L^2[0,1].$ \qeddef
\end{theorem}
We then prove a full characterization of complete systems of weighted exponentials that form a wovenly complete system. Such a woven structure exists if and only if the Fourier coefficients of the generators of involved systems of weighted exponentials are ``highly overlapping" in the sense of Theorem \ref{Completeness_Characterization}. Here we say a system of weighted exponentials $\Ec(f)$ is \emph{complete} if  $\clspan\set{fe^{2\pi inx}\,|\,n\in \Z}=L^2[0,1].$ We say two complete systems of weighted exponentials, $\Ec(f)$ and $\Ec(g)$, form a \emph{wovenly complete system} if $\set{fe^{2 \pi i n\cdot x}}_{n \in J}\cup \set{ge^{2 \pi i n \cdot x}}_{n \in J^c}$ remains complete for every $J\subseteq \Z.$
 As a result, we obtain the following criterion that two complete systems of weighted exponentials fail to be wovenly complete.
\begin{theorem}
\label{necessary_woven_completeness}
    Let $f,g$ be two functions in $L^2[0,1]$. Then $\Ec(f)$ and $\Ec(g)$ do not form a wovenly complete system of weighted exponentials if $$\widehat{f}(n)\widehat{g}(n)=0 \quad \text{for all }n\in \Z$$
    where $\widehat{f}(n)$ denotea the $n$-th Fourier coefficient of $f.$
    
    In particular, any two systems of weighted exponentials generated by two trigonometric polynomials with disjoint spectra fail to form a wovenly complete system. \qeddef
\end{theorem}
 Finally, we present several counterexamples showing that certain “seemingly correct” conditions for systems of weighted exponentials do not necessarily imply that the corresponding woven system retains the same properties as the original systems of weighted exponentials. 

This paper is organized as follows. In Section \ref{preliminaries}, we review the necessary background, including known results and notations used throughout this paper. The relationships between systems of weighted exponentials, systems of regular translates and Gabor system at crtical density will be specified in this section as well. We then present our main results in Section \ref{main_reusults} including strengthened version of Theorem \ref{sufficient_woven_frames_special_case}, \ref{necessary_woven_completeness} and some notable examples. Finally, we apply results from Section \ref{main_reusults} to systems of regular translates and Gabor system at critical density and highlight some results in such settings in Section \ref{application}.

\section{Preliminaries}
\label{preliminaries}
Throughout this paper, $L^2[0,1]$ denotes the Hilbert space of all (Lebesgue) square-integrable functions defined on $[0,1]$. By $\ip{\cdot}{\cdot}$ we mean the inner product associated with $L^2[0,1].$ The definitions of Fourier transform and inverse Fourier transform we adopt here are
$$\mathcal{F}(f)(\xi)=\widehat{f}(\xi)=\int_{\R} f(t)e^{-2\pi it\xi}\,dt~\text{  and  }~\mathcal{F}^{-1}(f)(\xi)=\check{f}(\xi)=\int_{\R} f(t)e^{2\pi it\xi}\,dt,$$
respectively, while for each $n\in \Z$ the $n$-th Fourier coefficient is defined to be 
$$\widehat{f}(n)=\ip{f}{e_n}=\int_{0}^1fe^{-2\pi i n x}\,dx,$$
where $e_n(x)=e^{2\pi inx}.$ For any function $f\in L^2([0,1])$ and any subset $J\subseteq\Z$ we denote by $\Ec(f,J)$ the subsequence $\set{fe^{2\pi inx}}_{n\in J}$ of $\Ec(f)$.
In the case that $J=\Z$, we will simply write $\Ec(f)$ instead of $\Ec(f,\Z^d).$ 

We now define below the notion of woven system of weighted exponentials. Recall that a sequence $\set{x_n}_{n\in I}$ is \emph{minimal} if there exists some sequence $\set{y_n}_{n\in I}$ such that $\ip{x_n}{y_m}=\delta_{mn}$ for all $m,n\in I$ (for example, see \cite[Lemma 5.4]{Heil11}). We say a sequence is \emph{exact} if it is complete and minimal.
\begin{definition}
\label{def_woven_system}
Let $f,g\in L^2[0,1].$ 
\begin{enumerate}
\setlength\itemsep{0.5em}
    \item [\textup{(a)}] The woven system associated with $\Ec(f)$ and $\Ec(g)$, $\Wc(f,g)$, is the collection of all weavings associated with $\Ec(f)$ and $\Ec(g)$. That is, 
    $$\Wc(f,g)=\set{\Ec(f,J)\cup\Ec(g,J^c)\,|\,J\subseteq \Z}.$$
    For convenience, we will simply denote $\Ec(f,J)\cup\Ec(g,J^c)$ by $W(f,g,J).$ 
    \item [\textup{(b)}] We say that $\Wc(f,g)$ is a \emph{woven frame} (respectively \emph{a woven Bessel sequence}, \emph{a woven orthonormal basis}) for $L^2[0,1]$ if every weaving in $\Wc(f,g)$ is a frame (respectively Bessel sequence, woven orthonormal basis) for $L^2[0,1]$. 

    \item [\textup{(c)}] We say that $\Wc(f,g)$ is \emph{wovenly complete} (respectively \emph{wovenly minimal}, \emph{wovenly  exact}) for $L^2[0,1]$ if every weaving in $\Wc(f,g)$ is complete (respectively minimal, exact) in $L^2[0,1]$. \qeddef
\end{enumerate}

\end{definition}
A necessary condition for a woven system associated with $\Ec(f)$ and $\Ec(g)$ to possess any of the properties listed in \ref{def_woven_system} is that both $\Ec(f)$ and $\Ec(g)$ individually satisfy that property. The following theorem (for example, see \cite[Theorem 10.10]{Heil11}) gives a complete characterization of systems of weighted exponentials that possess these properties. Recall that we say a frame for $L^2[0,1]$ is a \emph{Riesz basis} if the removal of any element from it leaves an incomplete set. 

\begin{theorem} \label{OneWeight}
\setlength\itemsep{0.3em}
    For any function $f\in L^2([0,1])$ the following statements hold.
    \begin{enumerate}
        \item [\textup{(a)}] $\Ec(f)$ is complete in $L^2[0, 1]$ if and only if $f(t)  \neq  0$ for a.e. $t$. 
        \medskip
        \item [\textup{(b)}] $\Ec(f)$ is minimal in $L^2[0, 1]$ if and only if $1/f \in L^2[0, 1]$. In this case $\Ec(f)$ is exact, and its biorthogonal system is $\Ec(\tilde{f})$ where $\tilde{f} = 1/\overline{f}$.
        \medskip
        \item [\textup{(c)}] $\Ec(f)$ is a Bessel sequence in $L^2[0, 1]$ if and only if $f \in L^\infty[0, 1]$. In this case $|f(t)|^2 \leq B$ a.e., where $B$ is a Bessel bound. 
        \medskip
        \item [\textup{(d)}] $\Ec(f)$ is a frame for $L^2[0, 1]$ if and only if there exist $A, B > 0$ such that $A \leq |f(t)|^2 \leq B$ for a.e. $x$. In this case $\Ec(f)$ is a Riesz basis for $L^2[0, 1]$. 
        \medskip
        \item [\textup{(e)}] $\Ec(f)$ is an orthonormal basis for $L^2[0, 1]$ if and only if $|f(t)| = 1$ for a.e. $t$. 
    \end{enumerate}
    \qeddef
\end{theorem}

For each $k\in \Z$ let $T_k\colon L^2(\R)\rightarrow L^2(\R)$ be the shift operator defined by $(T_kg)(x)=g(x-k).$ A system of regular translates associated with $g\in L^2(\R)$ is the sequence $\set{T_{k}g}_{k\in \Z}.$ It is known that such a system can never be complete in $L^2(\R)$. In fact, no system of regular translates can span any closed subspace of $L^2(\R)$ that is closed under the Fourier transform (\cite{PY25}). Consequently, one of the major focus on systems of regular translates is determining what conditions allow a system of regular translates to possess desired approximation properties for ``its closed span". For more details, as well as more results on systems of translates, we refer to \cite[Section 10.4]{Heil11}.

Let $g\in L^2(\R)$. The \emph{fiberization} of $g$, denoted by $\Phi_g (\xi)$, is the $1$-periodic function $$\Phi_g (\xi)=\sum_{k\in \Z} |\widehat{g}(\xi+k)|^2  \qquad \text{for $\xi\in \R.$}$$
Note that $\Phi_g$ is finite for almost every $\xi\in \R$. Let $\mathcal{O}_g$ be the set of zeros of $\Phi_g$.
The \emph{fiberization map} $\Psi_g$ associated with $g\in L^2(\R)$ is an isometry from $\clspan\set{T_{k}g}_{k\in \Z}$ onto $H_{\Psi_g}=\set{f\in L^2([0,1])\,|\,f=0 \text{ on }  \Oc_g}$, defined by 
\begin{equation}
\label{fiberization_map}
\Psi_g\Bigparen{\,\sum_{k\in \Z} c_kT_kg\,}(\xi)\Eq \widehat{c}(\xi)\,\Phi_g^{1/2}(\xi),\qquad c=(c_k)_{k\in \Z}\in c_{00}(\Z),
\end{equation}
where $\widehat{c}(\xi)=\sum_{k\in \Z} c_ke^{2\pi ik\xi}$ denotes the Fourier transform of $c.$ Consequently, the fiberization map transforms a system of regular translates into a system of weighted exponentials. It follows that $\set{T_{k}g}_{k\in \Z}$ possesses certain property if and only if $\Ec(\Phi_g^{1/2})$ does as well.

Finally, we briefly illustrate the relationship between systems of weighted exponentials and Gabor systems at critical density. For detailed explanations, we refer to \cite[Section 11.7]{Heil11}.
Let $g \in L^2(\R)$. The Zak transform is the isometry from $L^2(\R)$ onto $L^2[0,1]^2$ defined by
    \begin{equation*}
        Zg(x, \xi) \EQ \sum_{j \in \Z} g(x-j) e^{2 \pi i j \xi}, \quad   (x, \xi) \in [0,1]^2,
    \end{equation*}
    where this series converges unconditionally in the norm of $L^2[0,1]^2$. The Gabor system associated with $g$ at critical density $\Gc(g,1,1)$ is the sequence $\set{M_nT_kg}_{n,k\in\Z}$, where $M_n\colon L^2(\R)\rightarrow L^2(\R)$ is the modulation operator defined by $(M_ng)(\xi)=e^{2\pi in\xi}g(\xi).$ It is known that, for each $n,k\in\Z$, 
    \begin{equation}
    \label{Zak_transform}
    \bigparen{Z(M_nT_kg)}(x,\xi)= e^{2\pi inx}e^{2\pi ik\xi}Z_g(x,\xi)\quad \text{for a.e. }x,\xi\in\R.
    \end{equation}
    Therefore, every Gabor system in $L^2(\R)$ at critical density is isometrically isomorphic to a two-dimensional system of weighted exponentials in $L^2[0,1]^2$. Although Theorem \ref{OneWeight} is stated only for $L^2[0,1]$, its results extend to $L^2[0,1]^d$ for any $d>1$ with only minor modifications to the proofs. As a result, we can characterize Gabor systems at critical density with certain properties via Theorem \ref{OneWeight}.


\section{Main Results}
\label{main_reusults}
In this section, we prove our main results and present several notable examples. Although we only provide proofs for $L^2[0,1]$, all results extend straightforwardly to $L^2[0,1]^d$. 

\subsection{Woven Bessel sequences}
We begin with woven Bessel sequence. Since the arguments are straightforward, we summarize the results in the following proposition and omit the proofs.

\begin{proposition}\label{charac_woven_Bessel}
Let $f$ and $g$ be functions in $L^2[0,1].$
\begin{enumerate}
\setlength\itemsep{0.5em}
    \item [\textup{(a)}] $\Wc(f,g)$ is a woven Bessel sequence if and only if both $f$ and $g$ are in $L^{\infty}[0,1]$.
    \item [\textup{(b)}] Let $\set{f_n}\inN$ and $\set{g_n}\inN$ be two sequences in $L^2[0,1]$ that converge to $f$ and $g$, respectively. Assume that both of $\sup_n\norm{f_n}_{L^\infty[0,1]}$ and $\sup_n\norm{g_n}_{L^\infty[0,1]}$ are finite. Then $\Wc(f,g)$ is a Woven Bessel sequence.
\end{enumerate}
\qeddef
\end{proposition}
We remark that the assumption that $\sup_n\norm{f_n}_{L^\infty[0,1]}$ and $\sup_n\norm{g_n}_{L^\infty[0,1]}$ are finite in Proposition \ref{charac_woven_Bessel} is necessary. For example,  let $f_n(t) = (t^{\frac{1}{4}} + \frac{1}{n} )^{-1}$ and $g_n(t) = 1$, for all $n \in \mathbb{N}$. Then $f_n\rightarrow t^{-1/4}$ and $g_n\rightarrow 1$ in $L^2[0,1]$, respectively. But $W(t^{-1/4},1)$ is not a woven Bessel sequence in $L^2[0,1].$ 

\subsection{Woven orthonormal bases} By Theorem \ref{OneWeight}, a system of weighted exponentials $\Ec(f)$ is an orthonormal basis for $L^2[0,1]$ if and only if $|f(t)|=1$ for almost every $t\in [0,1].$ This unique feature allows us to prove a complete characterization of woven orthonormal basis as follows. 

\begin{theorem}
\label{ONB_Characterization} Let $f$ and $g$ be functions in $L^2[0,1].$ Then 
$\Wc(f, g)$ is a woven orthonormal bases in $L^2[0,1]$ if and only if $|f(t)| = |g(t)| = 1$ a.e., and $f \overline{g}(t) = e^{i \theta}$ for some $\theta \in \mathbb{R}$. 
\end{theorem}
\begin{proof}
($\Rightarrow$) Assume $\Wc(f,g)$ is a woven orthonormal basis. Since $\mathcal{E}(f) = \Wc(f, g, \mathbb{Z})$ and $\mathcal{E}(g) = \Wc(f,g, \emptyset)$ are  orthonormal bases, we have that $|f| = |g| = 1$ a.e. by Theorem \ref{OneWeight} (e). Moreover, we have
\begin{align*}
0 \EQ \ip{fe_m}{ge_n}  \EQ \int_{0}^{1}f \overline{g}(t) e^{2 \pi i (m - n)t} dt \EQ \ip{f\overline{g}}{e_{m - n}}, \quad \text{for all}\ m \neq n, 
\end{align*}
where we have used the fact that $\Wc(f, g, J)$ is an orthonormal basis for every $J \subset \mathbb{Z}$. By the uniqueness of the Fourier transform, it follows that $f \overline{g}$ equals to some constant with modulus $1$.  

\medskip
($\Leftarrow$)
Conversely, assume $|f| = |g| = 1$ a.e. and $f \overline{g} = e^{i \theta}$. By Theorem \ref{OneWeight}, this implies that $\mathcal{E}(f)$ and $\mathcal{E}(g)$ are orthonormal bases. Next, let $J \subset \mathbb{Z}$ be a subset. We will show that $\Wc(f,g,J)$ is a complete orthonormal sequence. To this end, for any $m \in J$ and $n \in J^c$ we compute
\begin{align*}
\ip{fe_m}{ge_n} \EQ \int_{0}^{1} f \overline{g}(t)e^{2 \pi i (m - n)t} dt \EQ e^{i \theta}\ip{e_m}{e_n} \EQ 0.  
\end{align*}
Likewise, for $m, m^{'}\in J$ with $m \neq m^{'}$, we have
\begin{align*}
\ip{fe_m}{fe_{m^{'}}} \EQ \int_{0}^{1}|f(t)|^2 e^{2 \pi i (m - m^{'})t}dt \EQ \ip{e_m}{e_{m^{'}}} \EQ 0.     
\end{align*}
The same argument shows that $\ip{ge_n}{ge_{n^{'}}} = 0$, for all $n, n^{'} \in J^c$ with $n\neq n'$. This proves that $\Wc(f, g , J)$ is an orthonormal sequence. To prove completeness, assume that $h \in L^2[0,1]$ is such that
\begin{align*}
\ip{h}{fe_m} \EQ 0 \EQ \ip{h}{ge_n}, \quad \text{for all } m \in J \text{ and } n \in J^c.  
\end{align*}
It follows that
\begin{align*}
h \overline{f} \in \clspan\{e_n\}_{n \in J^c} \quad \text{and} \quad h \overline{g}  \in \clspan\{e_m\}_{m \in J}.
\end{align*}
Note that since $f, g \in L^{\infty}[0,1]$, we have $h\overline{f}, h \overline{g} \in L^2[0,1]$ . Finally, by the orthogonality of $h \overline{f}$ and $h \overline{g}$, we obtain
\begin{align*}
0=\int_{0}^{1}|h(t)|^2 f \overline{g}(t) \hspace{0.05cm} dt \EQ e^{i \theta}\int_{0}^{1}|h(t)|^2 \hspace{0.05cm} dt. \quad 
\end{align*}
Therefore, $h$ must be the zero function. It follows that $\Wc(f,g , J)$ is an orthonormal basis. Thus, $\Wc(f,g)$ is a woven orthonormal basis.
\end{proof}

We have seen, in the example following Proposition \ref{charac_woven_Bessel}, that woven Bessel sequences (in fact, even woven frames) are not stable under limiting process. The next proposition shows that this is not case for woven orthonormal bases. Since the proof is straightfoward, we omit the details.

\begin{proposition}
Let $\set{f_n}\inN$ and $\set{g_n}\inN$ be sequences in $L^2[0,1]$ such that $\Wc(f_n, g_n)\inN$ is a woven orthonormal bases for all $n\in\N$. Assume that $f_n$ and $g_n$ converge to $f$ and $g$ in $L^2[0,1]$, respectively. Then $\Wc(f,g)$ is a woven orthonormal basis. \qeddef
\end{proposition}

\subsection{Woven completeness and minimality} 
Next, we shift our focus to woven completeness and minimality.

It was revealed in the proof of Theorem \ref{ONB_Characterization} that two systems of weighted exponentials form a wovenly complete system if the signs of two generators are consistent. The proof is similar to the second half of the proof of Theorem \ref{ONB_Characterization}, so we omit it here.
\begin{proposition}\label{Completeness_SSign}
    Let $f$ and $g$ be functions in $ L^\infty[0, 1]$. Assume that  either $f\overline{g} > 0$ a.e. or $f \overline{g} < 0$ a.e.. Then $\Wc(f, g)$ is wovenly complete in $L^2[0, 1]$.
\end{proposition}
Proposition \ref{Completeness_SSign} provides a convenient sufficient condition for woven completeness. Nonetheless, as shown in the following example, it is not necessary. Recall that the Kadets $\frac{1}{4}$-Theorem (\cite{Kadets64}) states that if $\set{\lambda_n}_{n\in\mathbb{\Z}}$ is a sequence of real numbers satisfying $\sup_{n\in\mathbb{Z}}|\lambda_n-n|<\frac{1}{4},$ then $\set{e^{2\pi i\lambda_nt}}_{n\in\mathbb{Z}}$ is a Riesz basis for $L^2[0,1].$ 
\begin{example}
\label{non_real_form_woven_Riesz}
    Let $f(t) = e^{2 \pi i t / 5}$ and $g(t) = e^{2 \pi it / 6}$. By Kadets $1/4$-Theorem (\cite{Kadets64}), $\Wc(f, g, J)$ is not only complete but a Riesz basis for $L^2[0, 1]$ for any subset $J\subseteq\Z$. However,
       $ f\overline{g} \Eq e^{2 \pi i t/30}.$
    \qeddef
\end{example}

The multiplication by $e^{-2\pi int}$ is equivalent to shifting the spectrum by $n$ positions. In the notion of woven systems, we consider all possible spectral shifts of the generators involved in the weaving process. Consequently, for a system to be wovenly complete, the spectra of the generators must exhibit a significant degree of overlap, ensuring that exponential term is not absent in any weaving regardless of how the spectrum is shifted. We confirm that this is true in the following sense. We first establish a required lemma.

\begin{lemma}\label{Lemma_Completeness}
    Let $f \in L^1[0, 1]$ and let $J \subseteq \Z$ be any subset. If $\hat{f}(n) = 0$ for all $n \in J$, then
    \begin{equation*}
        f \in \overline{\text{span}}\set{e_n}_{n \in J^c},
    \end{equation*}
    in $L^1[0, 1]$. 
\end{lemma}
\begin{proof}
 Let $\set{\varphi_k}_{k \in \N} \subseteq C^\infty[0, 1]$ be an approximate identity such that $f_k = f \ast \varphi_k$ converges to $f$ in $L^1[0,1]$. Since $\hat{f}(n) = 0$ for all $n\in J$, we have $\widehat{f_k}(n) = \hat{f}(n) \widehat{\varphi}_k(n)=0 $ for all $n\in J.$ Consequently, $f_k \in \overline{\text{span}}\set{e_n}_{n \in J^c}$ in $L^2[0, 1]$ for all $k\in \N$. As a result, we have that  $f_k \in \overline{\text{span}}\set{e_n}_{n \in J^c}$ in $L^1[0, 1]$ by Cauchy–Bunyakovsky–Schwarz inequality,  Finally, since $f_k$ converges to $f$ in $L^1[0, 1]$ and $\overline{\text{span}}\set{e_n}_{n \in J^c}$ is closed in $L^1[0, 1]$, it follows that $f \in \overline{\text{span}}\set{e_n}_{n \in J^c}$ in $L^1[0, 1]$.
\end{proof}

\medskip
We are now ready to prove the main result of this subsection, a complete characterization of woven completeness. 

\medskip
\begin{theorem}\label{Completeness_Characterization}
    Let $f$ and $g$ be functions in $L^2[0, 1]$. Then $\Wc(f, g)$ is wovenly complete in $L^2[0, 1]$ if and only if both $f$ and $g$ are nonzero a.e. and for every nonzero $p \in L^2[0, 1]$, there does not exist any $J \subseteq \Z$ such that
    \begin{equation*}
        pf \in \overline{\text{span}}\set{e_n}_{n \in J} \qquad \text{and} \qquad pg \in \overline{\text{span}}\set{e_n}_{n \in J^c},
    \end{equation*}
    in $L^1[0, 1]$.
\end{theorem}
\begin{proof}
    Suppose there exist a nonzero $p \in L^2[0, 1]$ and $J \subseteq \Z$ such that
    \begin{equation*}
        pf \in \overline{\text{span}}\set{e_n}_{n \in J} \qquad \text{and} \qquad pg \in \overline{\text{span}}\set{e_n}_{n \in J^c},
    \end{equation*}
    in $L^1[0, 1]$. Let $\set{f_k}_{k \in \N} \subseteq \text{span}\set{e_n}_{n \in J}$ be a sequence such that $\norm{f_k - pf}_{1} \to 0$ as $k \to \infty.$ Since the Fourier transform $\mathcal{F}: L^1[0, 1] \to c_0(\Z)$ is bounded, we have $\widehat{f_k}(n) \to \widehat{pf}(n)$ for each $n \in \Z$. It follows that $\widehat{pf}(n) = 0$ for all $n \in J^c$. The same argument applied to $pg$ yields $\widehat{pg}(n) = 0$ for all $n \in J$. Let $I = (-J)^c$. Note that for each $n \in I$ we have $-n \in J^c$. Consequently, for each $n\in J$ we have
    \begin{equation*}
        \ip{\overline{p}}{fe_n} \Eq \ip{\overline{pf}}{e_n} \Eq \overline{\ip{pf}{e_{-n}}} \Eq 0.
    \end{equation*}
    Likewise, for eahc $n \in I^c$ we have $-n \in J$. Therefore,
    \begin{equation*}
        \ip{\overline{p}}{ge_n} \Eq \ip{\overline{pg}}{e_n} \Eq \overline{\ip{pf}{e_{-n}}} \Eq 0.
    \end{equation*}
    Therefore, the weaving $\Wc(f, g, I)$ is not complete. 

    Conversely, assume that $\Wc(f, g)$ is not complete. If either $f$ or $g$ vanishes on a set of positive measure, then either $\Wc(f,g,\Z)$ or $\Wc(f,g,\emptyset)$ is not complete. Otherwise, suppose that $f$ and $g$ are nonzero a.e.. By the incompleteness of $\Wc(f, g)$ there exists a nonzero function $h \in L^2[0, 1]$ and $I \subseteq \Z$ such that
    \begin{equation*}
        \ip{h}{fe_n} \Eq 0, \quad \forall n \in I \qquad \text{and} \qquad \ip{h}{ge_n} \Eq 0, \quad \forall n \in I^c.
    \end{equation*}
    Let $p = \overline{h}$ and let $J = (-I)^c$. For any $n \in I$ we have
    \begin{equation*}
        0=\ip{h}{fe_n} \Eq \ip{\overline{pf}}{e_n} \Eq \overline{\ip{pf}{e_{-n}}},
    \end{equation*}
    which implies $\widehat{pf}(n) = 0$ for all $n \in -I$. Applying Lemma \ref{Lemma_Completeness}, we conclude
    \begin{equation*}
        pf \in \overline{\text{span}}\set{e_n}_{n \in (-I)^c} \Eq\overline{\text{span}}\set{e_n}_{n \in J}
    \end{equation*}
    in $L^1[0, 1]$. Arguing similarly, we obtain
    \begin{equation*}
        pg \in \overline{\text{span}}\set{e_n}_{n \in -I} \Eq \overline{\text{span}}\set{e_n}_{n \in J^c}
    \end{equation*}
    in $L^1[0, 1]$. This concludes the proof. 
\end{proof}

\begin{remark}
Let $f$ and $g$ be nonzero functions in $L^2[0,1]$ such that $f\overline{g}>0$ (or $f\overline{g}<0$).
    By Proposition \ref{Completeness_SSign} and Theorem \ref{Completeness_Characterization}, we observe that the spectra of $f$ from the spectra of $g$ are never disjoint in any weighted $L^2[0,1]$, provided the weight belongs to $L^2[0,1].$\qeddef
\end{remark}

We present an illustrative example demonstrating an application of Theorem \ref{Completeness_Characterization} below.
\begin{example}
\label{application_woven_comple}
    Let $f(t) = e^{2\pi i k t} + a e^{2\pi i m t}$ and $g(t) = e^{2\pi i k t} + b e^{2\pi i m t}$
    for some distinct integers $k, m$ and some complex numbers $a, b$ with $|a| \leq 1$ and $|b| \leq 1$. We will show that $\Wc(f, g)$ is wovenly complete. Note that $\Wc(f, g)$ is wovenly complete if and only if $\Wc(e^{2\pi i \ell t}f, e^{2\pi i \ell t}g)$ is wovenly complete for any $\ell \in \Z$. Consequently, we may assume that $k = 0$, so that $f = 1 + a e^{2\pi i m t}$ and $g = 1 + b e^{2\pi i m t}$.

    Suppose, to the contrary, that there exists a nonzero $p \in L^2[0, 1]$ such that $pf$ and $pg$ have disjoint spectra. Write $c_n = \widehat{p}(n)$ for $n \in \Z$. Since $f = 1 + a e^{2\pi i m t}$ and $g = 1 + b e^{2\pi i m t}$, we have, for every $n \in \Z$,
    \begin{equation*}
        \widehat{pf}(n) \EQ c_n + a\, c_{n-m}
        \qquad \text{and} \qquad
        \widehat{pg}(n) \EQ c_n + b\, c_{n-m}.
    \end{equation*}
    Because $pf$ and $pg$ have disjoint spectra, for each $n \in \Z$ at least one of
    $\widehat{pf}(n)$ and $\widehat{pg}(n)$ vanishes, that is,
    \begin{equation*}
        c_n + a\, c_{n-m} \EQ 0
        \qquad \text{or} \qquad
        c_n + b\, c_{n-m} \EQ 0.
    \end{equation*}
    In either case $c_n = -a\, c_{n-m}$ or $c_n = -b\, c_{n-m}$, so
    \begin{equation*}
        |c_n| \Le \max\set{|a|, |b|}\, |c_{n-m}| \Le |c_{n-m}|
        \qquad \text{for all } n \in \Z.
    \end{equation*}
    Iterating the above gives $|c_n| \leq |c_{n-km}|$ for every $k \in \N$. Since $p \in L^2[0, 1] \subseteq L^1[0, 1]$, the Riemann--Lebesgue lemma gives $c_{n-km} \to 0$ as $k \to \infty$, so $c_n = 0$.
    As $n \in \Z$ was arbitrary, $p = 0$ by uniqueness of Fourier coefficients, which is a contradiction. Thus there does not exist such a function $p$, and by Theorem \ref{Completeness_Characterization}, $\Wc(f, g)$ is wovenly complete. \qeddef
\end{example}

We now turn to the part of woven minimality. By Theorem \ref{OneWeight} (b), minimality is a stronger property than completeness in the setting of systems of weighted exponentials. Consequently, one might expect that woven minimality would imply woven completeness. We confirm this in the following results, provided that two involved systems of weighted exponentials are frames for $L^2[0,1].$ In fact, we prove that woven completeness is equivalent to a weaker type of woven minimality, which we define below.

\begin{definition}\label{L2_Independence_Def} Let $f$ and $g$ be functions in $L^2[0,1]$. We say that $\Wc(f,g)$ is \emph{wovenly $\ell^2$-minimal}, if for any $J\subseteq \Z$ and any sequence of scalars $(c_n)_{n\in\Z}\in \ell^2(\Z)$ the equation $$\sum_{n\in J}c_nfe_n+\sum_{n\in J^c} c_nge_n=0$$
holds only if $c_n=0$ for all $n\in\Z.$
\end{definition}

\begin{lemma}\label{Structural_Lemma}
    Let $f$ and $g$ be functions in $L^2[0,1]$ such that 
    $\Ec(f)$ and $\Ec(g)$ are frames for $L^2[0,1]$. Then the following statements hold:
    \begin{enumerate}
        \item [(a)] $\Wc(f, g)$ is wovenly complete if and only if $\Wc(1/f, 1/g)$ is wovenly complete. 
        \medskip
        \item [(b)] $\Wc(f, g)$ is wovenly minimal if and only if $\Wc(1/f, 1/g)$ is wovenly minimal.
    \end{enumerate}
\end{lemma}
\begin{proof} 
    (a) Assume $\Wc(f, g)$ is wovenly complete. For any $h\in L^2[0,1]$ and any subset $J\subseteq \Z$ we have
     \begin{equation*}
            \Bigip{ h}{\frac1g e_n}  \EQ \Bigip{ \frac{h}{\overline{fg}}}{ f e_n}  \EQ 0, 
        \end{equation*}
   for all $n \in J$. And
   \begin{equation*}
            \Bigip{ h}{\frac1f e_n}  \EQ \Bigip{ \frac{h}{\overline{fg}}}{ g e_n}  \EQ 0, 
        \end{equation*}
    for all $n \in J^c$. Note that $h/(\overline{fg})\in L^2[0,1]$ by Theorem \ref{OneWeight}. Statement (a) is now established.
    
        \medskip
        (b) Suppose that $\Wc(f, g)$ is wovenly minimal. Then for any subset $J \subseteq \Z$ there exists a sequence $\set{h_{n,J}}_{n\in\Z}$ such that  
\begin{equation*}
            \ip{h_{n,J}}{fe_k} \EQ \delta_{nk} \qquad \text{and} \qquad \ip{h_{n,J}}{ge_\ell} \EQ \delta_{n\ell},
        \end{equation*}
 for all $k\in J$ and $\ell\in J^c$ . Using the same technique as in part (a), it follows that
 \begin{equation*}
            \Bigip{\overline{fg}h_{n,J}}{ \frac{1}{f}e_\ell} \EQ \ip{h_{n,J}}{ge_\ell} \EQ \delta_{n\ell},
        \end{equation*}
        and  
        \begin{equation*}
         \Bigip{\overline{fg}h_{n,J}}{ \frac{1}{g}e_k} \EQ \ip{h_{n,J}}{fe_k} \EQ \delta_{nk},
        \end{equation*}
 for all $k\in J$ and $\ell\in J^c$. We then conclude statement (b). 
\end{proof}

We now prove the equivalence between woven $\ell^2$-minimality and wovenly completeness as follows. Recall that a sequence $\set{x_n}\inN$ in $L^2[0,1]$ is minimal if and only if $x_n\notin \clspan\set{x_m}_{m\in\N,\,m\neq n}$ for all $n\in \N.$
\begin{proposition}\label{L2_Independence_Prop}
    Let $f$ and $g$ be functions in $L^2[0,1]$ such that 
    $\Ec(f)$ and $\Ec(g)$ are frames for $L^2[0, 1]$. Then $\Wc(f, g)$ is wovenly $\ell^2$-minimal if and only if $\Wc(f, g)$ is wovenly complete. 
\end{proposition}
\begin{proof}
    Suppose $\Wc(f, g)$ is not wovenly $\ell^2$-minimal. Then there exists $J \subseteq \Z$ and a nonzero sequence $c = (c_n) \in \ell^2(\Z)$ such that 
    \begin{equation*}
        \sum_{n \in J} c_n fe_n \plus \sum_{n \in J^c} c_n g e_n \EQ 0.
    \end{equation*}

    Let $h_1 = \sum_{n \in I} c_n e_n$, $h_2 = \sum_{n \in I^c} (-c_n)e_n$, and set $p = fh_1 = gh_2$. Then, 
    \begin{equation*}
        \frac{p}{f} \in \clspan\set{e_n}_{n \in I} \qquad \text{and} \qquad \frac{p}{g} \in \clspan\set{e_n}_{n \in I^c}.
    \end{equation*}
    By Theorem \ref{Completeness_Characterization} we conclude that $\Wc(1/f, 1/g)$ is not wovenly complete. By Lemma \ref{Structural_Lemma}, it follows that $\Wc(f, g)$ is not wovenly complete as well. The converse can be proved similarly. 
\end{proof}

\begin{example}
    Fix any $n,m\in\Z$. We claim that $\Wc(ae_n+be_m,ce_n+de_m)$ is wovenly $\ell^2$-minimal for any nonzero scalars $a,b,c,d$ with $|a|>|b|$ and $|c|>|d|$. Note that $\Wc(ae_n+be_m,ce_n+de_m)$ is wovenly minimal if and only if $\Wc(\frac{1}{a}(ae_n+be_m),\frac{1}{c}(ce_n+de_m))$ for any nonzero scalars $C$ and $D$. Furthermore, using the shifting technique as in Example \ref{application_woven_comple}, we may assume that $n=0$. Therefore,  $\Wc(ae_n+be_m,ce_n+de_m)$is wovenly minimal if and only if $\Wc(1+\alpha e_{m-n},1+\beta e_{m-n})$ wovenly minimal, where $|\alpha|=|\frac{b}{a}|<1$ and $|\beta|=|\frac{d}{c}|<1$. By Theorem \ref{OneWeight}, $\Ec(1+\alpha e_{m-n})$ and $\Ec(1+\beta e_{m-n})$ are frames for $L^2[0,1].$ We have shown that $\Wc(1+\alpha e_{m-n},1+\beta e_{m-n})$ is wovenly complete in Example \ref{application_woven_comple}. By Proposition \ref{L2_Independence_Prop}, it follows that $\Wc(1+\alpha e_{m-n},1+\beta e_{m-n})$ is wovenly $\ell^2$-minimal.
\end{example}

Since the generic minimality is a stronger property than $\ell^2$-minimality in the setting of frames of weighted exponentials, we obtain the following corollary. We leave open the question of whether woven minimality always implies woven completeness for systems of weighted exponentials. In fact, it remains unknown whether woven completeness implies woven minimality in the setting of frames of weighted exponentials.
\begin{corollary}\label{RB_Minimality_Implies_Complete}
     Let $f$ and $g$ be functions in $L^2[0,1]$ such that 
    $\Ec(f)$ and $\Ec(g)$ are frames for $L^2[0,1]$. If $\Wc(f, g)$ is wovenly minimal, then $\Wc(f, g)$ is wovenly complete (and hence exact). \qeddef
\end{corollary}

As shown in the following example, the converse of Corollary \ref{RB_Minimality_Implies_Complete} does not generally hold without assuming that the two systems of weighted exponentials involved are frames.

\begin{example}
    Let
        $f(t) = e^{2 \pi it} + e^{-2 \pi i t} = 2\cos (2\pi t)$ and   $g(t) = e^{2 \pi it} - e^{-2 \pi i t} = 2i\sin (2\pi t)$.
  By Theorem \ref{OneWeight}, both $\Ec(f)$ and $\Ec(g)$ are not frames for $L^2[0,1]$. In particular, $\Wc(f, g)$ is not wovenly minimal. Nevertheless, $\Wc(f, g)$ is wovenly complete by Example \ref{application_woven_comple}. \qeddef
\end{example}

Lemma \ref{Structural_Lemma} shows that, for frames of weighted exponentials, woven completeness and woven minimality are invariant under replacing the weights by their reciprocals. In analogy with one-weighted setting, one might hope that woven minimality could be characterized by the simple reciprocal integrability conditions, for example, that the individual minimality of $\Ec(f)$ and $\Ec(g)$ would force woven minimality of $\Wc(f, g)$. This proves to be false. In fact, even if both $\Ec(f)$ and $\Ec(g)$ are exact, there is no guarantee that $\Wc(f, g)$ would be minimal. We provide an explicit counterexample below.

\begin{example}
    We first define
    \begin{equation*}
        f(t) \EQ 
        \begin{cases}
            t^{-1/4}, & t \in [0, 1/4), \\
            1, & t \in [1/4, 1/2) \cup [3/4, 1), \\
            (t-1/2)^{-1/4}, & t \in [1/2, 3/4),
        \end{cases}
    \end{equation*}
    and
    \begin{equation*}
        g(t) \EQ 
        \begin{cases}
            t^{1/4} e^{-2 \pi i t}, & t \in [0, 1/4), \\
            e^{-2 \pi i t}, & t \in [1/4, 1/2) \cup [3/4, 1), \\
            (t-1/2)^{1/4} e^{-2 \pi i (t-1/2)}, & t \in [1/2, 3/4).
        \end{cases}
    \end{equation*}
    A direct computation shows that $\Ec(f)$ and $\Ec(ge_1)$ are exact systems in $L^2[0, 1]$. We will show that the weaving $\Wc(f, ge_1, 2\Z)$ is not minimal. Equivalently, $\Ec(f,2\Z)\cup\Ec(g,2\Z)$ is not minimal.
    
    Next, we define the $\frac12$-periodic functions
    \begin{equation*}
        \tau(t) \EQ 
        \begin{cases}
            t^{1/2}, & t \in [0, 1/4), \\
            0, & t \in [1/4, 1/2),
        \end{cases}
        \qquad \text{and} \qquad 
         \gamma(t) \EQ 
        \begin{cases}
            e^{2 \pi i t}, & t \in [0, 1/4), \\
            0, & t \in [1/4, 1/2).
        \end{cases}
    \end{equation*}
   Note that $\tau f = \gamma g$ in $L^2[0, 1]$. Since $\set{\sqrt{2}e_{2n}}_{n \in \Z}$ is an orthonormal basis for $L^2[0, 1/2]$ there exists a sequence $(c_n)_{n \in \Z} \in \ell^2(\Z)$ such that 
    \begin{equation*}
        \gamma \EQ \sum_{n \in \Z} c_n e_{2n},
    \end{equation*}
    with convergence of the series in the norm of $L^2[0, 1]$. Because $g \in L^\infty[0, 1]$, we obtain that
    \begin{equation*}
        \gamma g \EQ \sum_{n \in \Z} c_n e_{2n} g
    \end{equation*}
    with convergence of the series in the norm of $L^2[0, 1]$. Let $n_0 \in \Z$ be such that $c_{n_0} \neq 0$. Then 
    \begin{equation*}
        \tau f - \bigparen{\sum_{n \neq n_0} c_n e_{2n}} g \EQ c_{n_0} e_{2n_0} g.
    \end{equation*}
    By \cite[Corollary 2.5]{PY25}, $\tau f\in \clspan{\set{\Ec(f, 2\Z)}}.$ Thus, $e_{2n_0}g \in \overline{\text{span}}\bigset{\Ec(f, 2\Z) \cup \Ec(g, 2\Z \setminus \set{2n_0})}$. This implies that $\Wc(f, ge_1, 2\Z)$ is not minimal.\qeddef
\end{example}

\subsection{Woven frames} We conclude this section by presenting results regarding woven frames. 

In Proposition \ref{Completeness_SSign}, we showed that two functions generate a woven complete system if their product is either strictly positive or strictly negative almost everywhere on $[0,1]$. We further extend this results to the case of frames of weighted exponentials.

\begin{theorem}
\label{woven_Riesz_suff_condi}
     Let $f,g\in L^2[0,1]$ be such that $\Ec(f)$ and $\Ec(g)$ are frames. If $\frac{g}{f}$ is either strictly positive or strictly negative almost everywhere on $[0,1]$, then $\Wc(f,g)$ is a woven frame for $L^2([0,1]).$
\end{theorem}
\begin{proof}
We first observe that for any subset $J\subseteq \Z$ and any nonzero scalars $C,D$ we have 
$$\sum_{n\in J}\bigabs{\ip{h}{Cfe_n}}^2+\sum_{n\in J^c}\bigabs{\ip{h}{Dge_n}}^2\Le \max(C,D) \Bigparen{\sum_{n\in J}\bigabs{\ip{h}{fe_n}}^2+\sum_{n\in J^c}\bigabs{\ip{h}{ge_n}}^2},$$
and
$$\sum_{n\in J}\bigabs{\ip{h}{Cfe_n}}^2+\sum_{n\in J^c}\bigabs{\ip{h}{Dge_n}}^2\Ge \min(C,D) \Bigparen{\sum_{n\in J}\bigabs{\ip{h}{fe_n}}^2+\sum_{n\in J^c}\bigabs{\ip{h}{ge_n}}^2},$$
for all $h\in L^2[0,1]$. Consequently, $\Wc(f,g)$ is a woven frame if and only if $\Wc(Cf,Dg)$ is a woven frame for $L^2[0,1]$ for some nonzero scalars $C,D.$ Therefore, without loss of generality, we may assume $\frac gf>0$ a.e. on $[0,1]$.
Furthermore, since $\Ec(f)$ and $\Ec(g)$ are frames, there exists some $\delta_1,\delta_2>0$ such $\delta_2\geq \frac{g}{f}\geq\delta_1$ almost everywhere on $[0,1].$ By choosing $C$ large enough that $\norm{\frac{g}{Cf}}_{\infty}<1$, we may assume that $\norm{\frac{g}{f}-1}_{\infty}<1$. 

    Now for a fixed subset $J \subseteq \Z$ we define the orthogonal projection $P_J: \ell^2(\Z) \to \ell^2(\Z)$ associated with $J$ by 
    
    $$ (P_Jc)_n = 
    \begin{cases}
        c_n, & n \in J \\
        0, & n \in J^c
    \end{cases}.$$ 
    and define the synthesis operator $T^*_J: \ell^2(\Z) \to L^2[0, 1] $ associated with $\Wc(f,g,J)$ by
    $$ T^*_Jc = \sum_{n \in J} c_n fe_n + \sum_{n \in J^c} c_nge_n,$$
    where $c=(c_n)_{n\in\Z}.$
    We also define the ``inverse Fourier transform" $S\colon \ell^2(\Z)\rightarrow L^2[0,1]$ by $$S(c)=\sum_{n\in\Z}c_ne_n$$
    For notational convenience, for any $h\in L^\infty([0,1])$ we denote by $M_f$ the multiplication operator on $L^2[0,1]$ defined by $M_fh=fh.$ 
    With these notations, we obtain the following expression 
    \begin{align*}
        T^*_J &= M_fS^{-1} P_J + M_g S^{-1} P_{J^c}\\ &\Eq M_f S^{-1} + M_{g-f} S^{-1}P_{J^c} \\
        &= (M_f S^{-1}) (I + (M_f S^{-1})^{-1} M_{g-f} S^{-1} P_{J^c}),
    \end{align*}
    where $I$ is the identity operator on $L^2[0,1].$
Since $T_J^*$ maps the standard orthonormal basis for $\ell^2(\Z)$ to $\Wc(f,g,J)$, it remains to show that $T_J^*$ is an invertible bounded linear operator. Note that both $M_f$ and $S$ are invertible, so it is equivalent to show that $$\norm{I+(M_f \Fc^{-1})^{-1} M_{g-f} \Fc^{-1} P_{J^c}-I}\Eq \norm{(M_f \Fc^{-1})^{-1} M_{g-f} \Fc^{-1} P_{J^c}}<1.$$ We then compute 
    \begin{align}
    \label{Key_estimate}
        \norm{I+(M_f \Fc^{-1})^{-1} M_{g-f} \Fc^{-1} P_{J^c}-I} &= \norm{(M_f \Fc^{-1})^{-1} M_{g-f} \Fc^{-1} P_{J^c}} \\
        &= \norm{\Fc M_{(g-f)/f} \Fc^{-1} P_{J^c}}  \\
        & \leq \norm{\Fc} \norm{M_{(g-f)/f}} \norm{\Fc^{-1}} \norm{P_{J^c}} \\
        & \leq \norm{M_{(g-f)/f}} \\
        & = \left\|\frac{g}{f}-1\right\|_\infty  \\
        & < 1.
    \end{align}
 It follows that $I + (M_f \Fc^{-1})^{-1} M_{g-f} \Fc^{-1} P_{J^c}$ is an invertible bounded linear operator. Thus, $\Wc(f,g,J)$ is a frame (in fact, a Riesz basis) for $L^2([0,1])$ for any $J\subset \Z.$
\end{proof}

We highlight some notable implications of Theorem \ref{sufficient_woven_frames_special_case} in the following remark. 
\begin{remark} (a) We then obtain Theorem \ref{sufficient_woven_frames_special_case} as an immediate corollary of Theorem \ref{woven_Riesz_suff_condi}.
\smallskip

(b) As pointed out in Example \ref{non_real_form_woven_Riesz}, Theorem \ref{woven_Riesz_suff_condi} offers a simple sufficient condition, but it is not necessary.

\smallskip
(c)  If $f\in L^2[0,1]$ is such that $\Ec(f)$ is a frame for $L^2[0,1]$, then $\Wc(f,1/\overline{f})$ is a woven Riesz basis. Recall that the famous Wiener's Lemma (for example, see \cite[page 217]{Kat04}) states that if $f$ is a continuous function such that $f(t)\neq 0$ for every $t\in [0,1]$. Then $\frac 1f$ has absolutely summable Fourier coefficients if $f$ does. However, the distribution of the spectrum $\frac 1f$ over $\Z$ is usually unknown. For example, is the spectra of $\frac 1f$ related to $f$ in some sense? Let $\sigma(f)$ and $\sigma(1/f)$ be the spectra of $f$ and $1/f$, respectively. Note that $\sigma(1/\overline{f})=-\sigma(1/f)$. By Theorem \ref{Completeness_Characterization} and Theorem \ref{woven_Riesz_suff_condi}, we see that the $\sigma(f)$ and $-\sigma(1/f)$ must be highly overlapping in the sense of Theorem \ref{Completeness_Characterization}. For instance, there does not exist any $f$ such that $\sigma(f)$ and $-\sigma(1/f)$ are disjoint, provided that $f$ has absolutely summable Fourier coefficients and never vanishes on $[0,1]$.

\medskip
(d) The key step in the proof of Theorem \ref{woven_Riesz_suff_condi} is the satisfaction of inequality $$\norm{M_{(g-f)/f}}=\norm{(g-f)/f}_{L^\infty[0,1]}<1.$$ As a result, we can also derive the following perturbation result using the same proof: ``For any $f\in L^2[0,1]$ such that $\Ec(f)$ is a frame for $L^2[0,1]$ the system $\Wc(f,f+h)$ is a woven frame for any $h\in L^2[0,1]$ with $\norm{h}_{L^\infty[0,1]}<\norm{f}_{L^\infty[0,1]}$.\qeddef 
\end{remark} 

We conclude this section by establishing the stability of woven Riesz bases under Paley–Wiener type perturbations. For other perturbation results regarding frames in Hilbert spaces, as well as in Banach spaces, we refer to \cite{CC97}, \cite{CC98} and \cite{CH97}.

\begin{theorem}
 Let $f \in L^2[0,1]$ be such that $\Ec(f)$ is a frame with frame bounds $A$ and $B$. Let $\lambda, \mu >0$ be such that $\mu + \lambda\sqrt{B}< \frac{A}{2\sqrt{B}}.$ Assume that $g\in L^2[0,1]$ satisfies 
\begin{align}
\label{pertubation_criterion}
\Bignorm{\sum_{n \in \mathbb{Z}} c_n (f  - g)e_n}_{L^2[0,1]} \Le \lambda\,\Bignorm{\sum_{n \in \mathbb{Z}} c_n fe_n}_{L^2[0,1]} + \mu\Bigparen{\sum_{n \in \mathbb{Z}}|c_n|^2}^{\frac{1}{2}},
\end{align}
for all $(c_n)_{n\in\Z}\in \ell^2(\Z)$.
Then $\Wc(f,g)$ is a woven frame.    
\end{theorem}
\begin{proof}
Note that Equation (\ref{pertubation_criterion}) implies that the linear operator $R:\ell^2(\Z)\rightarrow L^2[0,1]$ defined 
$$R\bigparen{(c_n)_{n\in\Z}}\Eq \sum_{n \in \mathbb{Z}} c_n (f  - g)e_n$$
is bounded and $\norm{R}\leq (\lambda \sqrt{B}+\mu).$ Now fix $J \subset \mathbb{Z}$. For any  $h \in L^2[0,1]$ we define 
\begin{align}
\label{perturbation_equation_I}
           S(h) \EQ \sum_{n \in J}|\ip{h}{fe_n}|^2 + \sum_{n \in J^c}|\ip{h}{ge_n}|^2.   
\end{align}
For notational convenience, we denote by $S_1(h)$ and $S_2(h)$ the first term and the second term on the right-hand side of Equation (\ref{perturbation_equation_I}), respectively.

We estimate $S_2(h)$ first. A straightforward computation shows that
$$S_2(h)\EQ \sum_{n \in J^c}|\ip{h}{fe_n}|^2 + \sum_{n \in J^c}|\ip{h}{(g-f) e_n}|^2 + \sum_{n \in J^c}2 \text{Re}\bigparen{\ip{h}{f e_n} \overline{\ip{h}{(g-f) e_n}}},$$
where $\text{Re}\bigparen{\ip{h}{f e_n} \overline{\ip{h}{(g-f) e_n}}}$ denotes the real part of $\ip{h}{f e_n} \overline{\ip{h}{(g-f) e_n}}$. By Cauchy–Bunyakovsky–Schwarz inequality, we obtain
\begin{align}
\begin{split}
\text{Re}\bigparen{\ip{h}{f e_n} \overline{\ip{h}{(g-f) e_n}}}&\Le \Bigparen{\sum_{n \in J^c}|\ip{h}{(g-f) e_n}|^2}^{\frac{1}{2}}\Bigparen{\sum_{n \in J^c}|\ip{h}{fe_n}|^2}^{\frac{1}{2}}\\
&\Le (\lambda \sqrt{B}+\mu)\sqrt{B}\norm{h}_{L^2[0,1]}^2,
\end{split}
\end{align}
where we have used the fact $\Bigparen{\sum_{n \in J^c}|\ip{h}{(g-f) e_n}|^2}^{\frac{1}{2}}\Le (\lambda \sqrt{B}+\mu) \norm{h}_{L^2[0,1]}.$ Consequently, $\Wc(f,g,J)$ is a Bessel sequence with an upper frame bound $B+2(\lambda\sqrt{B}+\mu)\sqrt{B}+(\lambda\sqrt{B}+\mu)^2.$
On the other hand, since
$$S_2(h)\Ge \sum_{n \in J^c}|\ip{h}{fe_n}|^2 + \sum_{n \in J^c}|\ip{h}{(g-f) e_n}|^2-2(\lambda \sqrt{B}+\mu)\sqrt{B} \norm{h}_{L^2[0,1]},$$
it follows that 
\begin{align}
\begin{split}
S(h)&\Ge \sum_{n\in\Z}|\ip{h}{fe_n}|^2+\sum_{n \in J^c}|\ip{h}{(g-f) e_n}|^2-2(\lambda \sqrt{B}+\mu)\sqrt{B} \norm{h}_{L^2[0,1]}\\
&\Ge  \sum_{n\in\Z}|\ip{h}{fe_n}|^2 -2(\lambda \sqrt{B}+\mu)\sqrt{B} \norm{h}_{L^2[0,1]}.
\end{split}
\end{align}
Thus, if $A>2\sqrt{B}(\lambda\sqrt{B}+\mu)$, then $\Wc(f,g,J)$ would be a frame for $L^2[0,1]$ with frame bounds $A-2\sqrt{B}(\lambda\sqrt{B}+\mu)$ and $B+2(\lambda\sqrt{B}+\mu)\sqrt{B}+(\lambda\sqrt{B}+\mu)^2$ 
\end{proof}
\begin{remark} \label{generalization1}
Fix $a>0.$ By slightly modifying the proof, we can obtain an analogous of Theorem \ref{OneWeight} for the systems of weighted exponentials $\set{ge^{2\pi iant}}_{n\in\Z}$ in $L^2[0,1/a].$  As a result, all results presented above apply equivalently to systems of weighted exponentials in $L^2[0,1/a]$.
\end{remark}

\section{applications}
\label{application}
Finally, we present several applications of our main results in the setting of systems of regular translates and Gabor systems with critical density. All proofs presented below extend straightforwardly to higher-dimensional cases.
\subsection{Systems of regular translates}
The fiberization map (\ref{fiberization_map}) transforms a system of regular translates in $L^2(\R)$ into a system of weighted exponentials in $L^2[0,1]$ with a real-valued, non-negative generator. Consequently, for any $J\subseteq \Z$ the union $\clspan\set{T_kf}_{k\in J}\cup\clspan\set{T_kg}_{k\in J^c}$ is complete (resp. minimal, a frame) in its closed span in $L^2(\R)$ if and only if 
$\clspan\set{\Phi_f e_k}_{k\in J}\cup\clspan\set{\Phi_g e_k}_{k\in J^c}$ is complete (resp. minimal, a frame) in its closed span in $L^2[0,1].$ We accordingly make use our main results to obtain the following corollaries.

Let $g\in L^2(\R)$ and assume that $f\in \clspan\set{T_kg}_{k\in\Z}.$ 
It is known that $\clspan\set{T_kf}_{k\in\Z}$ is $\clspan\set{T_kg}_{k\in\Z}$ if and only if $\text{supp}(\widehat{f})=\text{supp}(\widehat{g})$ (for example, see \cite[Corollary 2.4]{BDR94}). However, it is not known whether the closed span of the ``partial translates" of two distinct generators would compensate each other and still cover $\clspan\set{T_kf}_{k\in\Z}$.
We provide a affirmative answer to this question below, provided that the fiberization of $g$ is nonzero almost everywhere.

\begin{corollary}
\label{woven_completeness_system_translates}
Let $g\in L^2(\R)$ be such that $\Phi_g\neq 0$ almost everywhere on $[0,1]$. Assume that $f,h$ are functions in $ \clspan\set{T_{k}g}_{k\in \Z}$ with $\text{supp}(\widehat{f})=\text{supp}(\widehat{h})=\text{supp}(\widehat{g})$. Then for any $J\subseteq\Z$ the set $\clspan\set{T_{k}f}_{k\in J} \cup \clspan\set{T_{k}g}_{k\in J^c}$
is complete in $\clspan\set{T_{k}g}_{k\in \Z}$ .
\end{corollary}
\begin{proof}
By \cite[Corollary 2.4]{BDR94}, it follows that $\Phi_f$ and $\Phi_h$ are also nonzero almost everywhere on $[0,1].$ The statement then follows by Proposition \ref{Completeness_SSign} that $\Wc(\Phi_f,\Phi_h)$ are wovenly complete in $L^2[0,1].$
\end{proof}

Using Theorem \ref{woven_Riesz_suff_condi}, we obtain the frame-version of Corollary \ref{woven_completeness_system_translates}.
\begin{corollary}
\label{woven_completeness_system_translates}
Let $g\in L^2(\R)$ be such that $A\leq \norm{\Phi_g}_{L^\infty[0,1]}\leq B$ for some positive constants $A$ and $B$. Assume that $f\in \clspan\set{T_{k}g}_{k\in \Z}$ satisfies 
 $m\leq \norm{\Phi_f}_{L^\infty[0,1]}\leq M$ for some positive constants $m$ and $M$. Then for any $J\subseteq\Z$ the sequence $\set{T_{k}g}_{k\in J} \cup \set{T_{k}f}_{k\in J^c}$
is a frame for $\clspan\set{T_{k}g}_{k\in \Z}$. Furthermore, $\set{T_{k}g}_{k\in J} \cup \set{T_{k}f}_{k\in J^c}$ is $\ell^2$-minimal.
\end{corollary}
\begin{proof}
By Theorem \ref{OneWeight} and \cite[Corollary 2.4]{BDR94}, both $\Ec(\Phi_g)$ and $\Ec(\Phi_{f})$ are frames for $L^2[0,1].$
Moreover, $\Phi_g/\Phi_{f}$ is strictly positive almost everywhere on $[0,1].$ The statement then follows by Theorem \ref{woven_Riesz_suff_condi}. Furthermore, $\set{T_{k}g}_{k\in J} \cup \set{T_{k}f}_{k\in J^c}$ is $\ell^2$-minimal by Proposition \ref{L2_Independence_Prop}
\end{proof}
\begin{remark} Fix $a>0$ and let $g\in L^2(\R)$. We can define the \emph{a-fiberization} of $g$, denoted by $\Phi_g (\xi)$, to be the $a$-periodic function $\sum_{k\in \Z}|\widehat{g}(x-ak)|^2$. Modifying the definition of the fiberization map (\ref{fiberization_map}) accordingly, it follows that all results presented above apply equivalently to system of regular translates of the form $\set{T_{ak}g}_{k\in\Z}$ for some $a>0$ and $g\in L^2(\R)$ (See also Remark \ref{generalization1}). \qeddef
\end{remark}
\subsection{Gabor systems at critical density} To apply our main resuls to Gabor systems with critical density, we first recall the definition of \emph{Feichtinger algebra}. For more background and results on this topic, we refer to \cite{Gro01}.
\begin{definition}
Fix a nonzero Schwartz function $\psi\in S(\R)$. 
\begin{enumerate}\setlength\itemsep{0.5em}
    \item [\textup{(a)}] The \emph{short-time Fourier transform} of $f\in L^2(\R)$, denoted by $V_{\psi}f$, is the complex-valued measurable function on $\R^2$ defined by  $$V_{\psi}f(x,w)\Eq \overline{\ip{M_wT_x \psi}{f}} \Eq  \ip{f}{M_wT_x \psi}.$$
     \item [\textup{(b)}] The \emph{Feichtinger algebra} $M^{1}(\R)$ is the space consisting of 
      $f\in L^2(\R)$ for which $\norm{V_\psi f}_{L^{1}(\R^2)}$ is finite, i.e.,
     $M^{1}(\R)=\bigset{f \in L^2(\R)\,\big|\,  \norm{f}_{M^{1}(\R)}\Eq \norm{V_\psi f}_{L^{1}(\R^2)}<\infty}.$
  
    \end{enumerate}
\end{definition} 
It is known that $Zf(x,\xi)= e^{2\pi ix\xi}Z\widehat{f}(\xi,-x)$ for all $x,\xi\in \R$, provided that $f\in M^1(\R)$ (for example, see \cite[Proposition 8.2.2]{Gro01}). Moreover, for any $f\in L^2(\R)$ that is compactly supported in $[0,1]$ we have 
        $Zf(x,\xi) \EQ f(x)$ for all $(x,\xi)\in [0,1]^2.$ We accordingly obtain the following corollary.
\begin{corollary} Let $f,g\in M^1(\R)$ be functions be such that $\widehat{f}$ and $\widehat{g}$ are compactly supported in $[0,1].$  Assume that either  $\widehat{f}(\xi)\widehat{g}(\xi)>0$ or $\widehat{f}(\xi)\widehat{g}(\xi)<0$ for almost every $\xi\in [0,1]$. Then for any subset $J\subseteq \Z^2$ $\set{M_nT_k f}_{(n,k)\in J}\cup \set{M_nT_k g}_{(n,k)\in J^c}$
    is complete in $L^2(\R).$ \qeddef    
\end{corollary}
\section*{acknowledgement}
We thank Chris Heil for his helpful comments on this project. We also thank Thibaud Alemany for helpful conversations during the course of this project.

\end{document}